\documentclass[11pt,reqno]{amsart}
\usepackage{amsmath,amsthm,amssymb,mathtools}
\usepackage{geometry}
\usepackage{tikz}
\usetikzlibrary{arrows.meta, positioning}
\usepackage{float}
\usepackage{graphicx}
\usepackage{hyperref}
\usepackage{xcolor}
\usepackage{xspace}

\title[A quadratic form generalization of rational dinv]{A quadratic form generalization of rational dinv}
\author{Yifeng Huang\\$\,$\\{\tiny with appendix by} Kenny Lau}
\date{\today}

\newtheorem{theorem}{Theorem}[section]
\newtheorem{lemma}[theorem]{Lemma}

\theoremstyle{remark}
\newtheorem{remark}[theorem]{Remark}
\newtheorem*{remark*}{Remark}

\begin{document}

\begin{abstract}
We introduce a quadratic form $Q$ on the space of functions on the gap poset $G$ of the numerical semigroup $\langle a,b\rangle$. We prove combinatorially that when evaluated on the indicator function of an upward closed subset $D$, this quadratic form precisely recovers the Gorsky--Mazin $\mathtt{dinv}$ statistic of $D$, viewed as a Young subdiagram of $G$. Furthermore, we prove Theorem~\ref{thm:bilinear-sum} that when evaluated on a pair of subdiagrams of $G$, the symmetric bilinear form associated with $Q$ is equal to a novel cross-$\mathtt{dinv}$ statistic, which is nonnegative. Combining these, we prove the inequality
\[ Q(\mathbf{n})\geq \dfrac{1}{|G|}\,\|\mathbf{n}\|_\infty^2\]
if $\mathbf{n}$ is a real-valued decreasing function on $G$, showing an effective positive definiteness of $Q$ on the corresponding cone. Theorem~\ref{thm:bilinear-sum}, the main engine of the paper, was autoformalized in Lean/Mathlib by AxiomProver.
\end{abstract}

\maketitle

\section{Introduction}

\subsection{Background and Motivation}
The classical $q,t$-Catalan numbers $C_n(q,t)$, introduced by Garsia and Haiman \cite{GH96}, are a crown jewel of algebraic combinatorics, defined as the sum over Dyck paths of bi-weighted statistics $\mathtt{area}$ and $\mathtt{dinv}$. Gorsky and Mazin \cite{GM13, GM14} generalized this theory to the rational $(a,b)$-Catalan numbers $C_{a,b}(q,t)$ for coprime positive integers $a<b$, corresponding to rational Dyck paths in an $a \times b$ rectangle. Along with Oblomkov, Rasmussen, and Shende \cite{ORS18}, they established profound geometric significance of these combinatorial objects by connecting the generating functions of these paths to the HOMFLY-PT homology of the $(a,b)$-torus knot, and algebro-geometrically, to the cohomology of the compactified Jacobian (or Hilbert scheme of points) of the singular plane curve $x^a = y^b$. 

In connecting Dyck paths to geometry, the gap set $G$ of the numerical semigroup $\Gamma=\langle a,b\rangle$ plays an important role. The function
\[ g(x,y)=ab-ax-by\]
induces a bijection from
\begin{equation}\label{eq:gap-diagram}
    \{(x,y)\in \mathbb{Z}_{\geq 1}^2:g(x,y)>0\}
\end{equation}
to $G$, realizing $G$ as a Young diagram consisting of boxes in an $a\times b$ rectangle strictly below the diagonal; this Young diagram is where rational Dyck paths live. 

In this paper, we consider a quadratic form arising from the high-rank generalization of the geometric setting. In forthcoming work \cite{HJO26+} with Ruofan Jiang and Alexei Oblomkov concerning the geometry of the Quot scheme of points on the curve $x^a=y^b$, the construction of an affine cell decomposition of the moduli space motivates a novel quadratic form that recovers the $\mathtt{dinv}$ statistic and generalizes it to higher ranks. Its definition is algebraic:
\[
Q(\mathbf{n})=\sum_{i,j\in G}  K(j-i)\,n_{i}n_{j},\]
defined for $\mathbf{n}=(n_i)_{i\in G}\in \mathbb{R}^G$,
where
\[K(d)=\mathbf{1}_{d\geq 0}-\mathbf{1}_{d\geq a}-\mathbf{1}_{d\geq b}+\mathbf{1}_{d\geq a+b}.
\]
It is conjectured that despite the negative signs, $Q$ is positive definite on the cone 
\[ C_{\mathbb{R}}=\{\mathbf{n}\in \mathbb{R}^G:n_i\geq 0, n_j\geq n_i\text{ if }j-i\in \langle a,b\rangle\}.\]

The significance of this conjecture in the context of \cite{HJO26+} is two-fold. First, it is proposed that finding a positive formula for $Q(\mathbf{n})$ for integer-valued vectors $\mathbf{n}\in C_\mathbb{R}$ should lead to a notion of a high-rank $\mathtt{dinv}$ for nested Dyck paths. Second, the positive definiteness is of fundamental importance in the convergence of an infinite $q$-series. In \cite{HJO26+}, it is proved that for a finite field $\mathbb{F}_q$, the ``number'' of finite-cardinality modules over the ring $R_{a,b}=\mathbb{F}_q[[T^a,T^b]]$ (the germ of the plane curve singularity) weighted inversely by the size of the automorphism groups (a.k.a. the groupoid volume) is given by an infinite sum:
\[ \begin{aligned}
    &\mathtt{groupoid\text{-}vol}(\mathbf{FinMod}_{R_{a,b}})\coloneq\sum_{ \mathbf{FinMod}_{R_{a,b}}/{\sim}} \frac{1}{|\mathrm{Aut}_{R_{a,b}}(M)|} \\
    &=  \left(\prod_{n=1}^\infty (1-z^n)^{-1} \right)Z_{a,b}(z)\coloneqq \left(\prod_{n=1}^\infty (1-z^n)^{-1} \right) \sum_{\mathbf{n}\in \mathbb{Z}^G\cap C_{\mathbb{R}}} z^{Q(\mathbf{n})}(1+O_{\mathbf{n}}(z))\in [0,\infty],
\end{aligned}\]
where $z=q^{-1}$ and $O_{\mathbf{n}}(z)$ is an explicit polynomial in nonnegative integer coefficients with zero constant term. Hence, the volume of $\mathbf{FinMod}_{R_{a,b}}$ is finite if and only if $Q(\mathbf{n})$ is strictly positive definite on $C_\mathbb{R}$, in which case $Z_{a,b}(z)$ is well-defined as a formal power series in $z$, and converges for $|z|<1$. The authors of \cite{HJO26+} then discovered and conjectured an infinite product expression for $Z_{a,b}(z)$, which amounts to a brand new bi-infinite family of sum$=$product identities of Rogers--Ramanujan type. The first new examples of the sum$=$product conjecture are proven in forthcoming work of the author with  Ken Ono and Peter Paule.

In this paper, we prove the conjecture that $Q$ is positive definite on $C_{\mathbb{R}}$. To achieve this goal, we connect it to $\mathtt{dinv}$, and prove that the symmetric bilinear form associated with $Q$ is positive semi-definite on $C_\mathbb{R}$, by interpreting certain evaluations in terms of ``cross-$\mathtt{dinv}$s'' of Dyck paths, a non-negative quantity we define combinatorially.

\subsection{Definitions and Main Results}
We treat $G$ as a poset with $i\preceq j$ if and only if $j-i\in \langle a,b\rangle$. We also treat $G$ as a Young diagram whose set of cells is defined by \eqref{eq:gap-diagram}. We orient the Young diagram so that $+x$ points to the east and $+y$ points to the north, and thus the unique maximal element of $(G,\preceq)$ lives at the \emph{southwest} corner. A subset $D \subseteq G$ is a \emph{subdiagram} (or by abuse of terminology, a \emph{Dyck path}) if it is an upward closed subset of $(G,\preceq)$, or equivalently, a Young subdiagram in the usual sense. 

Gorsky and Mazin defined the statistic $\mathtt{dinv}(D)$ for a Dyck path $D$. For each cell $c\in D$, let the arm and leg lengths of $c$ relative to $D$ be
\[ \mathrm{arm}_D(c)=\max\{k : c - ka \in D\}, \quad \mathrm{leg}_D(c)=\max\{k : c - kb \in D\}.\]
Then
\[
\mathtt{dinv}(D)=\#\left\{c\in D: \frac{\mathrm{leg}_D(c)}{\mathrm{arm}_D(c)+1} < \frac{a}{b} < \frac{\mathrm{leg}_D(c)+1}{\mathrm{arm}_D(c)}\right\}.
\]
For a subdiagram $D$, let $\mathbf{1}_D=(1_{i\in D})_i$ be the indicator vector of $D$. Our first main result realizes $\mathtt{dinv}$ as a first special case of the quadratic form $Q$.

\begin{theorem}\label{thm:dinv-match}
    For any subdiagram $D \subseteq G$, we have $Q(\mathbf{1}_D) = \mathtt{dinv}(D)$. Moreover, if $D\neq \varnothing$ then $Q(\mathbf{1}_D)>0$.\footnote{The last assertion also follows from the symmetry of $C_{a,b}(q,t)$, but we will give a direct proof.}
\end{theorem}

In particular, this theorem implies that $Q(\mathbf{n})$ is nonnegative for a $0,1$-vector $\mathbf{n}\in C_\mathbb{R}$. To extend this positivity to the whole $C_{\mathbb{R}}$, we introduce a symmetric cross-\texttt{dinv} statistic $\mathtt{dinv}(D,E)$ for a pair of Dyck paths $D,E\subseteq G$. For a cell $c\in D\cap E$, define the small and large \emph{mixed cross hook slopes} as
\[ m_{D}^{E}(c)=\frac{\mathrm{leg}_{E}(c)}{\mathrm{arm}_{D}(c)+1} < M_{D}^{E}(c)=\frac{\mathrm{leg}_{E}(c)+1}{\mathrm{arm}_{D}(c)}.\]
Then define
\[ \mathtt{dinv}(D,E)=\frac{\#\left\{c\in D\cap E: m_D^E(c)<\dfrac{a}{b}<M_D^E(c)\right\}+\#\left\{c\in D\cap E: m_E^D(c)<\dfrac{a}{b}<M_E^D(c)\right\}}{2}.\]

By definition, it satisfies $0\leq \mathtt{dinv}(D,E)\leq |D\cap E|$ and $\mathtt{dinv}(D,D)=\mathtt{dinv}(D)$.

Our second main result recognizes the cross-\texttt{dinv} as a bilinear analogue of the classical \texttt{dinv}. Let $B(\mathbf{n},\mathbf{n'})$ be the unique symmetric bilinear form satisfying $B(\mathbf{n},\mathbf{n})=Q(\mathbf{n})$ (see Section~\ref{sec:bilinear-sum} for concrete formulas). 

\begin{theorem}\label{thm:bilinear-sum}
    For subdiagrams $D,E\subseteq G$, we have
    \[ B(\mathbf{1}_D,\mathbf{1}_E)=\mathtt{dinv}(D,E).\]
\end{theorem}

With little extra effort, Theorem~\ref{thm:bilinear-sum} settles the positive definiteness conjecture and reveals that the symmetric bilinear form, in fact, is already positive semi-definite on $C_{\mathbb{R}}$.

\begin{theorem}\label{thm:pos-def}
    For $\mathbf{n},\mathbf{n'}\in C_{\mathbb{R}}$, we have
    \( B(\mathbf{n},\mathbf{n'})\geq 0,\text{ so }Q(\mathbf{n})=B(\mathbf{n},\mathbf{n})\geq 0.\)
    Moreover, we have an effective bound
    \[ Q(\mathbf{n}) \ge \frac{1}{|G|} \|\mathbf{n}\|_{\infty}^2 \text{ for }\mathbf{n}\in C_{\mathbb{R}}. \]
\end{theorem}

As an application, to compute $Z_{a,b}(q)=\sum_{\mathbf{n}\in \mathbb{Z}^G\cap C_{\mathbb{R}}} q^{Q(\mathbf{n})}(1+O_{\mathbf{n}}(q))$ up to the $q^N$-coefficient, it suffices to sum over $\mathbf{n}\in \mathbb{Z}^G\cap C_{\mathbb{R}}$ such that
\[ \lVert \mathbf{n}\rVert_\infty \leq \sqrt{N|G|},\] 
a finite set of vectors. Hence, $Z_{a,b}(q)$ converges formally as a power series.

\begin{remark}\label{rmk:counterexample}
    It is possible that $B(\mathbf{n},\mathbf{n'})=0$ for $\mathbf{n},\mathbf{n'}\in C_\mathbb{R}\setminus \{0\}$. For example, if $(a,b)=(3,5)$, $D=G=\{1,2,4,7\}$, and $E=\{7\}$, then $B(\mathbf{1}_D,\mathbf{1}_E)=0$.
\end{remark}

Finally, we raise the question of whether the cross-$\mathtt{dinv}$ and the quadratic form $Q$ have richer connections to algebraic combinatorics. Naturally, systems of $n$ nested Dyck paths are in bijection with vectors in $C_\mathbb{R}$ with values in $\{0,1,\dots,n\}$:
\[ \mathbf{D}=(D_1\subseteq \dots \subseteq D_n) \leadsto \mathbf{n}_\mathbf{D}\coloneq \mathbf{1}_{D_1}+\dots+\mathbf{1}_{D_n}.\]
(Indeed, the inverse map is given by sending $\mathbf{n}$ to $\mathbf{D}$ such that $D_i=\{g:n_g\geq i\}$.) As such, it makes sense to define a ``high-rank'' \texttt{dinv} for nested Dyck paths:
\[ \mathtt{dinv}(\mathbf{D}):=Q(\mathbf{n}_\mathbf{D})=\sum_{i,j=1}^n \mathtt{dinv}(D_i,D_j),\]
viewing the classical \texttt{dinv} as the rank-$1$ special case. We wonder how it compares to existing notions of high-rank or nested \texttt{dinv}, e.g., \cite{LW08}, which appears to be different.

\section{Preliminaries}
Fix integers $1<a<b$ with $\gcd(a,b)=1$. From now on, we work on the infinite grid $\mathbb{Z}^2$. We exclusively treat $G$ as the set
\begin{equation*}
    G=\{(x,y)\in \mathbb{Z}_{\geq 1}^2:g(x,y)>0\}.
\end{equation*}
Recall the function $g:\mathbb{Z}^2\to \mathbb{Z}$ defined by $g(x,y)=ab-ax-by$. We call $g(x,y)$ the \emph{value} of the cell $(x,y)$. Though $g$ is not globally injective, we often use the value of a cell to refer to the cell itself if there is no ambiguity. For example, if $c\in G$, we use $c-ka$ to refer to the cell $k$ steps to the east of $c$, even when it is outside $G$ and even though there are infinitely many cells in $\mathbb{Z}^2$ with the same value. In other words, we treat $a,b$ as vectors: $a=(-1,0), b=(0,-1)$.

We will work with the following reformulation of $Q$:
\[
Q(\mathbf{n})=\sum_{i,j\in G} \Bigl(\mathbf{1}_{0\le g(j)-g(i)<a}-\mathbf{1}_{b\le g(j)-g(i)<a+b}\Bigr)n_i n_j.
\]
It motivates the following notation for $i,j\in \mathbb{Z}^2$:
\[ i\to j \text{ if and only if }0\leq g(j)-g(i)<a.\]
Since $a<b$, other than self-loops, arrows $i\to j$ can only point strictly from northwest to southeast, or vice versa. (Indeed, in any other case, the absolute difference $|g(j)-g(i)|$ would be at least $a$.)

Given a cell $i\in \mathbb{Z}^2$ and a row index $r\in \mathbb{Z}$, there are unique members $j,k\in \{(x,y)\in \mathbb{Z}^2:y=r\}$ such that $i\to j$ and $k\to i$. We denote these cells by 
\[ j=\mathrm{proj}_r(i) \text{ and } k=\mathrm{antiproj}_r(i),\]
called the \emph{projection} and the \emph{antiprojection} of $i$ onto the $r$-th row. The unique existence is simply a consequence of division by $a$ with remainder; in fact, $j$ is the lowest-valued cell in row $r$ with value at least $g(i)$, and $k$ is the highest-valued cell in row $r$ with value at most $g(i)$.

As an important property, if $c\in G$ and $r$ is a row of $G$ non-strictly to the south of $c$ (i.e., $1\leq r \le \mathrm{row}(c)$), then $\mathrm{proj}_r(c)$ is in $G$; see Figure~\ref{fig:proj}. To verify this, we check each condition in the definition of $G$. That $\mathrm{proj}_r(c) \in \mathbb{Z}_{\geq 1}^2$ is automatic as $r$ is a row of $G$ and $\mathrm{proj}_r(c)$ is to the southeast of $c$. That $g(\mathrm{proj}_r(c))>0$ is simply because $g(\mathrm{proj}_r(c))\ge g(c)$.

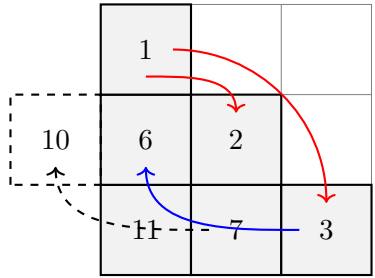
\begin{figure}[h]
    \centering
    \begin{tikzpicture}[scale=1.2]
        \draw[step=1cm,gray,very thin] (0,0) grid (3,3);
        \foreach \x/\y/\v in {0/0/11, 1/0/7, 2/0/3, 0/1/6, 1/1/2, 0/2/1} {
            \draw[thick, fill=gray!10] (\x,\y) rectangle (\x+1,\y+1);
            \node at (\x+0.5,\y+0.5) {\v};
        }
        \draw[thick, dashed] (-1,1) rectangle (0,2);
        \node at (-0.5,1.5) {10};

        \draw[->, thick, red] (0.5, 2.2) to[out=0,in=90] (1.5, 1.8); 
        \draw[->, thick, red] (0.8, 2.5) to[out=0,in=90] (2.5, 0.8); 
        \draw[->, thick, blue] (2.2, 0.5) to[out=180,in=270] (0.5, 1.2); 
        \draw[->, thick, dashed] (1.2, 0.5) to[out=180,in=270] (-0.5, 1.2); 
    \end{tikzpicture}
    \caption{For $(a,b)=(4,5)$, the projection of $1$ to other rows are $2$ and $3$. The projection of $3$ to the middle row (row 2) is $6$. The projection of $7$ to the middle row is $10$, which lies outside $G$.}
    \label{fig:proj}
\end{figure}

\section{Proof of Theorem \ref{thm:dinv-match}}
Though Theorem~\ref{thm:bilinear-sum} implies the main assertion of Theorem~\ref{thm:dinv-match}, we give a separate proof of Theorem~\ref{thm:dinv-match} here, as it motivates the proof method of Theorem~\ref{thm:bilinear-sum}.

\subsection{Restatement of $Q(D)$}
As the first step, we give a combinatorial reformulation of the evaluation $Q(D)\coloneq Q(\mathbf{1}_D)$. Let $U=U_D$ be the set of upper boundary cells just above $D\cup (\mathbb{Z}_{\geq 1}\times \mathbb{Z}_{\leq 0})$:
\begin{equation}\label{eq:U}
    U=U_D=\{(x,y): (x,y)\notin D \text{ and } (x,y-1) \in D \cup (\mathbb{Z}_{\geq 1}\times \mathbb{Z}_{\leq 0})\}.
\end{equation}

\begin{lemma}\label{lem:deficit}
We have 
    \[ Q(D)=|D|-\#\{(i,j)\in D\times U: i\to j\}.\]
\end{lemma}
\begin{proof}
By definition,
\[
Q(D)=\sum_{i,j\in D} \mathbf{1}_{i\to j}-\mathbf{1}_{i\to j-b}=\sum_{\substack{(i,j)\in D\times \mathbb{Z}^2\\i\to j}} \mathbf{1}_{j\in D}-\mathbf{1}_{j+b\in D}.
\]
For each $j\in \mathbb{Z}^2$, we claim that 
\begin{equation}\label{eq:reduction to boundary}
    \mathbf{1}_{j\in D}-\mathbf{1}_{j+b\in D}=\mathbf{1}_{j\in B}-\mathbf{1}_{j\in U},
\end{equation}
where $B=\mathbb{Z}_{\geq 1}\times \{1\}$ is the bottom row in the first quadrant.

We verify the claim in four cases:
\begin{itemize}
    \item $j\notin B\cup U$: Either $j,j+b\in D$ or $j,j+b\notin D$, getting $0=0$.
    \item $j\in B\setminus U$: Then $j\in B\cap D$, so $j\in D$ and $j+b\in D$, getting $1=1$.
    \item $j\in U\setminus B$: Then $j\in G\setminus D$ and $j+b\in D$ by the definition of $U$, getting $-1=-1$.
    \item $j\in U\cap B$: Then $j\in B\setminus D$, so $j\notin D$ and $j+b\notin D$, getting $0=0$.
\end{itemize}

It follows that
\[ Q(D)=\sum_{\substack{(i,j)\in D\times \mathbb{Z}^2\\i\to j}} \mathbf{1}_{j\in B}-\mathbf{1}_{j\in U} = \#\{(i,j)\in D\times B: i\to j\} - \#\{(i,j)\in D\times U: i\to j\}. \]

But for each $i\in D$, there is a unique $j\in B$, namely, $j=\mathrm{proj}_1(i)$, such that $i\to j$. We get
\[ \#\{(i,j)\in D\times B: i\to j\} = |D|,\]
completing the proof.
\end{proof}

\subsection{Setting up the bijection}
Let $N = \{(i,j)\in D\times U: i\to j\}$. By Lemma~\ref{lem:deficit}, $Q(D)=|D|-|N|$, so to find $Q(D)$, it suffices to find $|N|$. For $(i,j)\in N$, $i$ and $j$ cannot belong to the same row since $D\cap U=\varnothing$. We divide $N$ into two parts:
\begin{itemize}
    \item $N_b$: the set of \emph{blue arrows} $(i,j)\in N$ where $j$ is strictly north (thus strictly northwest) of $i$. 
    \item $N_r$: the set of \emph{red arrows} $(i,j)\in N$ where $j$ is strictly south (thus strictly southeast) of $i$. 
\end{itemize}

\begin{figure}[h]
    \centering
    \begin{tikzpicture}[scale=1.2]
        \foreach \x/\y in {0/0, 1/0, 2/0, 3/0, 0/1, 0/2} {
            \fill[yellow!30] (\x,\y) rectangle (\x+1,\y+1);
            \draw[thick] (\x,\y) rectangle (\x+1,\y+1);
        }
        \foreach \x/\y in {4/0, 1/1, 2/1} {
            \draw[thick, fill=white] (\x,\y) rectangle (\x+1,\y+1);
        }
        \foreach \x/\y/\v in {0/0/17, 1/0/13, 2/0/9, 3/0/5, 4/0/1, 0/1/10, 1/1/6, 2/1/2, 0/2/3} {
            \node at (\x+0.5,\y+0.5) {\v};
        }
        \draw[->, thick, blue] (3.5, 0.8) to[out=135,in=315] (1.8, 1.2); 
        \draw[->, thick, red]  (0.8, 2.2) to[out=315,in=135] (1.2, 1.8); 
    \end{tikzpicture}
    \caption{For $(a,b)=(4,7)$, if $D$ has row lengths $(4,1,1)$, then $U\cap G = \{1,6,2\}$. The arrow $(5,6)$ is the only blue arrow and $(3,6)$ is the only red arrow. Hence $Q(D)=|D|-1-1=4$.}
    \label{fig:arrows}
\end{figure}
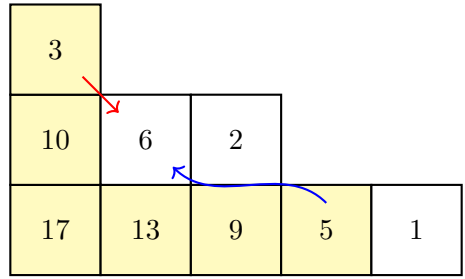

On the other hand, we divide the set of cells in $D$ violating the $\mathtt{dinv}$ condition into two parts. Recall the Gorsky--Mazin formulation of $\mathtt{dinv}$:
\[
\mathrm{dinv}(D)=\#\left\{c\in D: m_D(c) < \frac{a}{b} < M_D(c)\right\},
\]
where we denote the \emph{small hook slope} and \emph{large hook slope} respectively as:
\[ m(c)=\frac{\mathrm{leg}_D(c)}{\mathrm{arm}_D(c)+1},\quad M(c)=\frac{\mathrm{leg}_D(c)+1}{\mathrm{arm}_D(c)}.\]
Since $m(c)<M(c)$, the inequalities can be violated in two ways.
Define the set of \emph{blue cells} and \emph{red cells} respectively as:
\[ D_b=\left\{ c\in D: m(c)<M(c)\leq \frac{a}{b}\right\}, \quad D_r=\left\{ c\in D: \frac{a}{b}\leq m(c)<M(c)\right\}. \]
Since $a,b$ are coprime, the inequalities are strict. As a result, $\mathrm{dinv}(D)=|D|-|D_b|-|D_r|$. Intuitively, the hook of a blue cell is too flat, and the hook of a red cell is too steep. 

\begin{lemma}\label{lem:main-bijection}
    Let $a,b,D$ be fixed in the notation above. Then there are bijections
    \[ \Phi_b:N_b\to D_b,\quad \Phi_r:N_r\to D_r.\]
    Moreover, if $D\neq \varnothing$, let $d=\min D$. Then $d\in D\setminus(D_b\sqcup D_r)$. 
\end{lemma}

\begin{proof}
    See Figure~\ref{fig:main-bijection} for a visualization.

    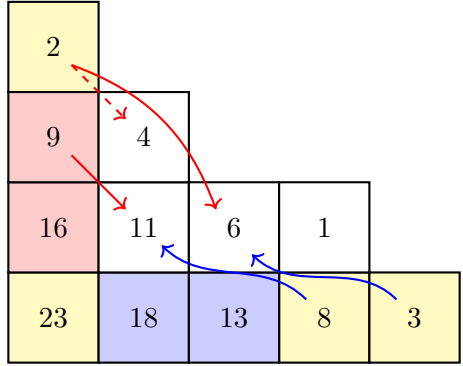
\begin{figure}[htbp]
    \centering
    \begin{tikzpicture}[scale=1.2]
        \foreach \x/\y in {0/0, 3/0, 4/0, 0/3} {
            \fill[yellow!30] (\x,\y) rectangle (\x+1,\y+1);
            \draw[thick] (\x,\y) rectangle (\x+1,\y+1);
        }
        \foreach \x/\y in {1/0, 2/0} {
            \fill[blue!20] (\x,\y) rectangle (\x+1,\y+1);
            \draw[thick] (\x,\y) rectangle (\x+1,\y+1);
        }
        \foreach \x/\y in {0/1, 0/2} {
            \fill[red!20] (\x,\y) rectangle (\x+1,\y+1);
            \draw[thick] (\x,\y) rectangle (\x+1,\y+1);
        }
        \foreach \x/\y in {1/1, 2/1, 3/1, 1/2} {
            \draw[thick, fill=white] (\x,\y) rectangle (\x+1,\y+1);
        }
        \foreach \x/\y/\v in {0/0/23, 1/0/18, 2/0/13, 3/0/8, 4/0/3, 
                              0/1/16, 1/1/11, 2/1/6, 3/1/1,
                              0/2/9, 1/2/4, 0/3/2} {
            \node at (\x+0.5,\y+0.5) {\v};
        }
        
        \draw[->, thick, blue] (3.3, 0.7) to[out=135,in=315] (1.7, 1.3); 
        \draw[->, thick, blue] (4.3, 0.7) to[out=135,in=315] (2.7, 1.2); 
        
        \draw[->, thick, red] (0.7, 3.3) to[bend left=25] (2.3, 1.7); 
        \draw[->, thick, red] (0.7, 2.3) to[out=315,in=135] (1.3, 1.7); 
        
        \draw[->, thick, red, dashed] (0.7, 3.3) to[out=315,in=135] (1.3, 2.7); 
        
    \end{tikzpicture}
    \caption{For $(a,b)=(5,7)$, if $D$ is the maximum hook, the blue cells are $13, 18$ (shaded blue) and the red cells are $9, 16$ (shaded red). The red arrow $(i,j)=(9,11)$ maps to $9$ via a parallel shift up to $(i',j')=(2,4)$.}
    \label{fig:main-bijection}
\end{figure}

    \medskip

    \noindent\textbf{Blue cells:} Define $\Phi_b(i,j)$ for $(i,j)\in N_b$ to be the cell $c$ at the intersection of the row of $i$ and the column of $j$. Because $j$ is to the northwest of $i$, $c$ is due west of $i$ and due south of $j$. Since $i \in D$ and $D$ is left-justified, $c \in D$. We have established a map $\Phi_b:N_b\to D_b$.

    To show $\Phi_b$ is injective with image $D_b$, we fix $c$ and reconstruct its preimage $(i,j)$. Since $j\in U$, the only choice of $j$ is $j=c-b\,(\mathrm{leg}_D(c)+1)$. As $i\to j$, the only choice of $i$ is $i=\mathrm{antiproj}_1(j)$, and it is valid (namely, $i\in D$) if and only if its value is at least the value of the rightmost cell of $D$ in the column of $c$, namely,
    \[ g(i)\geq g(c)-a\,\mathrm{arm}_D(c).\]
    By the characterization of $\mathrm{antiproj}_1(j)$ as the highest-valued cell in the bottom row with value at most $g(j)$, the inequality above is equivalent to
    \[ g(c)-b\,(\mathrm{leg}_D(c)+1) = g(j)\geq g(c)-a\,\mathrm{arm}_D(c), \]
    which simplifies to the defining condition of $D_b$.

    \medskip
    \noindent\textbf{Red cells:} Define $\Phi_r(i,j)$ for $(i,j)\in N_r$ as follows. Let $i'$ be the top cell of $D$ in the column of $i$. Let $j'=j-(i-i')$. Define $\Phi_r(i,j)$ to be the intersection $c$ of the column of $i'$ and the row of $j'$. It lies due south of $i'$ (which implies $c\in D$) and due west of $j'$. 

    To show $\Phi_r$ is injective with image $D_r$, we fix $c$ and reconstruct its preimage $(i,j)$. The only choice of $i'$ is the top cell of $D$ in the column of $c$, namely, $i'=c-b\,\mathrm{leg}_D(c)$. Since $i\to j$ and $g(j)-g(i)=g(j')-g(i')$, we have $i'\to j'$, so the only choice of $j'$ is $j'=\mathrm{proj}_{\mathrm{row}(c)}(i')$. As $j'$ is obtained by moving $j\in U$ up (nonstrictly), $j'$ is valid if and only if $j'\notin D$. A valid $(i',j')$ corresponds to a unique valid preimage $(i,j)$ by downward translation until $j'$ lands in $U$. 

    Hence, $\Phi_r$ is injective and $c$ lies in the image if and only if $j'\notin D$, or equivalently,
    \[ g(j')\leq g(c)-a\,(\mathrm{arm}_D(c)+1).\]
    By the characterization of $\mathrm{proj}_{\mathrm{row}(c)}(i')$, this is equivalent to
    \[g(c)-b\,\mathrm{leg}_D(c)=g(i')\leq g(c)-a\,(\mathrm{arm}_D(c)+1),\]
    which simplifies to the defining condition of $D_r$.

    \medskip
    \noindent\textbf{Minimal cell: } If $d$ were blue or red, then by the construction of $\Phi_b$ and $\Phi_r$, there is a cell $c\in D$ with $g(c)<g(d)$ (either to the due east of $d$ or due north of $d$), so $d$ cannot be minimal.
\end{proof}

\vspace{0.5em}

\begin{proof}[Proof of Theorem~\ref{thm:dinv-match}]
By Lemma~\ref{lem:deficit}, $Q(D)=|D|-|N_b|-|N_r|$. By Lemma~\ref{lem:main-bijection}, $|N_b|=|D_b|$ and $|N_r|=|D_r|$. It follows that $Q(D) = |D|-|D_b|-|D_r| = \mathtt{dinv}(D)$.
\end{proof}

\section{Proof of Theorem \ref{thm:bilinear-sum}}\label{sec:bilinear-sum}
We first recall some generalities about symmetric bilinear forms on $\mathbb{R}$-vector spaces. A bilinear form $B:V\times V\to \mathbb{R}$ on a vector space $V$ is said to be \emph{associated with} a quadratic form $Q:V\to \mathbb{R}$ if $B(v,v)=Q(v)$ for all $v\in V$. A bilinear form $B$ is said to be \emph{symmetric} if $B(v,v')=B(v',v)$ for all $v,v'\in V$. It is a standard fact that there is a unique symmetric bilinear form $B$ associated with $Q$, and it can be constructed in two equivalent way:
\begin{itemize}
    \item Set 
    \[ B(v,v')\coloneq \frac{1}{2}\Bigl(Q(v+v')-Q(v)-Q(v')\Bigr).\]
    \item Let $B'$ be any bilinear form associated with $Q$. Then we let $B$ be its symmetrization:
    \[ B(v,v')\coloneq \frac{1}{2}\Bigl(B'(v,v')+B'(v',v)\Bigr).\]
\end{itemize}

As is more practical for the purpose of Theorem~\ref{thm:bilinear-sum}, we shall use the second construction. Consider the bilinear form on $\mathbb{R}^G$:
\[ B'(\mathbf{n},\mathbf{n'})=\sum_{i,j\in G} \Bigl(\mathbf{1}_{0\le g(j)-g(i)<a}-\mathbf{1}_{b\le g(j)-g(i)<a+b}\Bigr)n_i n'_j.\]
Then $B'(\mathbf{n},\mathbf{n})=Q(\mathbf{n})$, so the symmetric bilinear form associated with $Q$ is given by
\[ B(\mathbf{n},\mathbf{n'})=\frac{1}{2}(B(\mathbf{n},\mathbf{n'})+B(\mathbf{n'},\mathbf{n})).\]
As before, for subdiagrams $D,E\subseteq G$, denote $B'(D,E)=B'(\mathbf{1}_D,\mathbf{1}_E)$ and $B(D,E)=B(\mathbf{1}_D,\mathbf{1}_E)$. Recall the definition of $U_D$ of any subdiagram $D$ from \eqref{eq:U}. For any two subsets $X,Y$ of $\mathbb{Z}^2$, define
\[ N(X,Y)\coloneq \{(i,j)\in X\times Y:i\to j\}.\]
\begin{lemma}
    \label{lem:B'}
    If $D,E$ are subdiagrams of $G$, then
    \[ B'(D,E)=|D|-|N(D,U_E)|.\]
\end{lemma}
\begin{proof}
    Similar to the proof of Lemma~\ref{lem:deficit}, 
    \[ B'(D,E)=\sum_{i\in D, j\in E} \mathbf{1}_{i\to j}-\mathbf{1}_{i\to j-b} = \sum_{\substack{(i,j)\in D\times \mathbb{Z}^2\\i\to j}} \mathbf{1}_{j\in E}-\mathbf{1}_{j+b\in E}.\]
    By \eqref{eq:reduction to boundary}, 
    \[ \mathbf{1}_{j\in E}-\mathbf{1}_{j+b\in E} = \mathbf{1}_{j\in B}-\mathbf{1}_{j\in U_E},\]
    where $B=\mathbf{Z}_{\geq 1}\times \{1\}$ is the bottom row, as before. Hence
    \[ B'(D,E) = \sum_{\substack{(i,j)\in D\times \mathbb{Z}^2\\i\to j}}\mathbf{1}_{j\in B}-\mathbf{1}_{j\in U_E} = N(D,B)-N(D,U_E).\]
    As in the proof of Lemma~\ref{lem:deficit}, $N(D,B)=|D|$, finishing the proof.
\end{proof}

\begin{proof}
    [Proof of Theorem~\ref{thm:bilinear-sum}]
    By preceding discussions, 
    \[ 2B(D,E)=B'(D,E)+B'(E,D)=|D|+|E|-|N(D,U_E)|-|N(E,U_D)|.\]
    Similar to the proof of Theorem~\ref{thm:dinv-match}, the key is to construct maps from $N(D,U_E)$ and $N(E,U_D)$ to $D\cup E$. Recall that an arrow $i\to j$ either points strictly to the northwest, points strictly to the southeast, or is a self-loop. We say an arrow is \emph{blue} if it points strictly to the northwest, and \emph{red} if it points strictly to the southeast or is a self-loop.\footnote{The treatment of self-loops is the ingredient new to Theorem~\ref{thm:dinv-match}.} Define $N_b(D,U_E)$ and $N_r(D,U_E)$ to be the set of blue arrows and the set of red arrows, respectively, and similarly define $N_b(E,U_D)$ and $N_r(E,U_D)$.

    Define $\Phi_{D,E}:N(D,U_E)\to D$ in a fashion similar to the proof of Lemma~\ref{lem:main-bijection}, namely:
    \begin{itemize}
    \item If $(i,j)$ is blue, then define $\Phi_{D,E}(i,j)$ to be the intersection of the row of $i$ and the column of $j$.
    \item If $(i,j)$ is red, then let $(i',j')$ be parallel to $(i,j)$, where $i'$ is the top cell of $D$ in the column of $i$. Define $\Phi_{D,E}(i,j)$ to be the intersection of the column of $i'$ and the row of $j'$.
    \end{itemize}  
    Define $\Phi_{E,D}:N(E,U_D)\to E$ similarly by symmetry.

    \medskip

    We now analyze the images of the blue and red arrows separately.

    \medskip

    \begin{itemize}
        \item Let 
        \[ D_b=\{c\in D\cap E:M^E_D(c)\leq a/b\}.\]
        Then we claim that $\Phi_{D,E}$ induces a bijection from $N_b(D,U_E)$ to $D_b$.
        
        Assume $\Phi_{D,E}(i,j)=c$ and $(i,j)$ is blue. Then $j$ is the unique element of $U_E$ in the column of $c$, and $i=\mathrm{antiproj}_{\mathrm{row}(c)}(j)$. Since $(i,j)$ is blue, $c$ is strictly to the south of $j$. This implies $c\in E$ and 
        \[j=c-b\,(\mathrm{leg}_E+1).\]
         
        The only candidate preimage $(i,j)$ is valid if and only if $i\in D$. This is equivalent to
        \[ g(c)-b\,(\mathrm{leg}_E+1)= g(j)\geq g(c)-a\, \mathrm{arm}_D(c),\]
        which is equivalent to $M^E_D(c)\leq a/b$, proving the claim.

        \item Let
        \[ D_r = (D\setminus E) \sqcup \{c\in D\cap E:m^D_E(c)\geq a/b\}.\]
        Then we claim that $\Phi_{D,E}$ induces a bijection from $N_r(D,U_E)$ to $D_r$.
        
        Assume $\Phi_{D,E}(i,j)=c$ and $(i,j)$ is red. Let $(i',j')$ be as in the construction of $\Phi_{D,E}(i,j)$. By construction, $i'$ must be the top cell of $D$ in the column of $c$, and $j'$ must be equal to $\mathrm{proj}_{\mathrm{row}(c)}(i')$. The only candidate $(i',j')$ corresponds to a valid preimage of $(i,j)$ if and only if $j'\notin E$. If $c\notin E$, this is automatic because $j'$ is nonstrictly to the east of $c$. If $c\in E$, the condition $j'\notin E$ is equivalent to 
        \[ g(c)-b\,\mathrm{leg}_D(c)=g(i')\leq g(c)-a\,(\mathrm{arm}_E(c)+1),\]
        which is equivalent to $m^D_E(c)\geq a/b$, finishing the proof of the claim.
    \end{itemize}

    By symmetry, defining
    \[ E_b=\{c\in D\cap E:M^D_E(c)\leq a/b\}\]
    and
    \[ E_r = (E\setminus D) \sqcup \{c\in D\cap E:m^E_D(c)\geq a/b\},\]
    then $\Phi_{E,D}$ restricts to bijections
    \[ N_b(E,U_D)\to E_b \text{ and }N_r(E,U_D)\to E_r.\]

    We are ready to prove the theorem. Denote 
    \[ \mathtt{dinv}^E_D=\#\{c\in D\cap E:m^E_D(c)<\frac{a}{b}<M^E_D(c)\}.\]
    Then note that
    \[ |D_b|+|E_r|=|E\setminus D|+(|D\cap E|-\mathtt{dinv}^E_D)=|E|-\mathtt{dinv}^E_D,\]
    and by symmetry
    \[ |D_r|+|E_b|=|D|-\mathtt{dinv}^D_E.\]
    Hence
    \begin{align*}
        2B(D,E)&=|D|+|E|-|N(D,U_E)|-|N(E,U_D)|\\
        &=|D|+|E|-|D_b|-|D_r|-|E_b|-|E_r|\\
        &=(|D|-|D_r|-|E_b|)+(|E|-|D_b|-|E_r|)\\
        &=\mathtt{dinv}^E_D + \mathtt{dinv}^D_E,
    \end{align*}
    as desired.
\end{proof}

\section{Proof of Theorem~\ref{thm:pos-def}}
To prove Theorem~\ref{thm:pos-def} using Theorem~\ref{thm:bilinear-sum}, it suffices to note that the cone $C_\mathbb{R}$ is generated by subdiagrams.
\begin{lemma}\label{lem:decompose}
    Let $\mathbf{n}\in C_{\mathbb{R}}$. Then there are an integer $0\leq k \leq |G|$, non-empty subdiagrams $D_1,\dots,D_k\subseteq G$, and positive real numbers $\lambda_1,\dots,\lambda_k$, such that
    \[ \mathbf{n}=\sum_{i=1}^k \lambda_i \mathbf{1}_{D_i}.\]
\end{lemma}
\begin{proof}
    We provide a canonical construction. Let $\{c_1<c_2<\dots<c_k\}$ be the set of nonzero coordinates of $\mathbf{n}=(n_g)_{g\in G}$, removing duplicates. Let $c_0=0$. For $1\leq i\leq k$, define $\lambda_i=c_i-c_{i-1}>0$ and
    \[ D_i=\{ g\in G:n_g\geq c_i \}.\]
    Because $\mathbf{n}\in C_{\mathbb{R}}$, $D_i$ is a subdiagram of $G$. (In fact we have a strictly nested sequence $0\subsetneq D_1\subsetneq \dots \subsetneq D_k\subseteq G$, though we don't need this fact.) It is straightforward to check that $\mathbf{n}=\sum_{i=1}^k \lambda_i \mathbf{1}_{D_i}$ as required.
\end{proof}

\begin{proof}[Proof of Theorem \ref{thm:pos-def}]
For $\mathbf{n},\mathbf{n'}\in C_\mathbb{R}$, write $\mathbf{n}=\sum_{i=1}^k \lambda_i \mathbf{1}_{D_i}$ and $\mathbf{n'}=\sum_{j=1}^l \mu_j \mathbf{1}_{E_j}$ using Lemma~\ref{lem:decompose}, where $\lambda_i,\mu_j>0$ and $D_i, E_j$ are nonempty subdiagrams. Then the bilinearity of $B$ implies
\begin{equation}\label{eq:bilinear-sum}
    B(\mathbf{n},\mathbf{n'})=\sum_{i,j} \lambda_i\mu_j \mathtt{dinv}(D_i, E_j)\geq 0.
\end{equation}

To establish the effective bounds, we apply \eqref{eq:bilinear-sum} to $\mathbf{n'}=\mathbf{n}$ and ignore cross terms, getting
\[ Q(\mathbf{n})\geq \sum_{i=1}^k \lambda_i^2 \mathtt{dinv}(D_i).\]
By the last assertion of Theorem~\ref{thm:dinv-match}, $\mathtt{dinv}(D_i)\geq 1$. This gives
\[ Q(\mathbf{n}) \geq \sum_{i=1}^k \lambda_i^2.\]
On the other hand, the $L_\infty$ norm of $\mathbf{n}$ is equal to $\sum_{i=1}^k  \lambda_i$, as both equal to the coordinate of $\mathbf{n}$ at the southwest corner of $G$: indeed, the fact that $\mathbf{n}\in C_\mathbb{R}$ implies that the maximum coordinate of $\mathbf{n}$ occurs at the southwest corner $f$ of $G$, and the fact that every nonempty diagram contains $f$ implies that $n_f=\sum_{i=1}^k \lambda_i$. Applying the Cauchy--Schwarz inequality, and recalling $k\leq |G|$, we get
\begin{align*}
    \|\mathbf{n}\|_\infty^2 &= \left(\sum_{i=1}^k  1\cdot \lambda_i\right)^2 
    \leq \left(\sum_{i=1}^k 1^2\right) \left(\sum_{i=1}^k \lambda_i^2\right)  
    \leq  |G|\, Q(\mathbf{n}),
\end{align*}
and the proof is complete.
\end{proof}

\section{Acknowledgements}\label{sec:ack}
The author thanks Ken Ono for helpful comments on an earlier draft, and Kenny Lau for implementing the Lean verification and providing the appendix. The author acknowledges that interactions with generative AIs (Gemini 3 Pro and ChatGPT 5.3) have contributed intellectually to the paper in the following ways:
\begin{itemize}
    \item AI generated Mathematica codes for numerous visualized examples in the style of Figure~\ref{fig:main-bijection}, from which the author managed to observe the key bijections in Lemma~\ref{lem:main-bijection}, including the definition of the blue and red arrows.
    \item When AI was given the proof of Theorem~\ref{thm:dinv-match} and then asked to prove $Q(\mathbf{n})\geq 0$ for $\mathbf{n}\in C_\mathbb{R}$ (part of Theorem~\ref{thm:pos-def}), the idea of bilinear form on nested subdiagrams $E\subseteq D$ was suggested, though AI initially thought that evaluations of the raw, non-symmetrized bilinear form $B'(\mathbf{1}_D,\mathbf{1}_E)$ and $B'(\mathbf{1}_E,\mathbf{1}_D)$ were already nonnegative. However, upon further query, AI was able to produce a counterexample: in the setting of Remark~\ref{rmk:counterexample}, we have $B'(\mathbf{1}_D,\mathbf{1}_E)=-1$.
\end{itemize}

\section{Appendix: AxiomProver, by Kenny Lau}\label{sec:AxiomProver}
At Axiom Math, we are developing AxiomProver, an AI system for mathematical research via formal proof.
As an early test case in this effort, we present this case study,
treating Theorem~\ref{thm:bilinear-sum} as an end-to-end formalization target.
We chose it because we believed that it is within reach of today's Mathlib, and it is the main engine of this paper.

This paper confirms that expectation: the proof is fully formalized in Lean/Mathlib
(see \cite{Lean,Mathlib2020})
and was produced by AxiomProver from a natural-language statement
of Theorem~\ref{thm:bilinear-sum}.
We now make precise what ``produced'' means, and describes the end-to-end pipeline we used.

\subsection*{Process}
The formal proofs provided in this work were developed and verified using \textbf{Lean 4.28.0}.
Compatibility with earlier or later versions is not guaranteed due to the evolving nature of the Lean 4 compiler and its core libraries.
The relevant files are all posted in the following repository:
\begin{center}
  \url{https://github.com/AxiomMath/quadratic-dinv}
\end{center}
The input files were
\begin{itemize}
  \item \texttt{quadratic-form-dinv.pdf} (not included): a first draft of the paper (which includes among other things Theorem~\ref{thm:bilinear-sum} and the proofs)
  \item a \texttt{task.md} that contains the single line
  \begin{quote}
    \slshape
    Prove Theorem 1.2
  \end{quote}
  \item a configuration file \texttt{.environment} that contains the single line
  \begin{quote}
    \slshape
    lean-4.28.0
  \end{quote}
  which specifies to AxiomProver which version of Lean should be used.
\end{itemize}
Given these three input files,
AxiomProver autonomously provided the following output files:
\begin{itemize}
  \item \texttt{problem.lean}, a Lean 4.28.0 formalization of the problem statement
which includes all the relevant definitions; and
  \item \texttt{solution.lean}, a complete Lean 4.28.0 formalization of the proof.
\end{itemize}

\medskip

The reader can verify these files using AXLE (Axiom Lean Engine), a tool developed by Axiom Math:
\begin{itemize}
\item Go to \url{https://axle.axiommath.ai/verify_proof}.
\item Select \texttt{lean-4.28.0} as the \texttt{environment}.
\item Put the content of \texttt{problem.lean} into the field \texttt{formal\_statement}.
\item Put the content of \texttt{solution.lean} into the field \texttt{content}.
\end{itemize}
This ensures that \texttt{solution.lean} contains correct proofs to theorems in \texttt{problem.lean} that were left open, and that it has not changed any of the definitions.

This section about AI is completely orthogonal to the rest of the paper, which is written without the use of AI except for the interactions described in Section~\ref{sec:ack}.
Indeed, a research paper is a narrative designed to communicate ideas to humans,
whereas a Lean files are designed to satisfy a computer kernel.

\end{document}